\documentclass[reqno]{amsart}
\usepackage{amsmath,amsthm,amstext,amssymb}
\usepackage{graphicx}
\usepackage{stmaryrd}
\usepackage{enumerate}

\newcounter{name}

\theoremstyle{plain}
\newtheorem{theorem}{Theorem}[section]
\newtheorem{lemma}[theorem]{Lemma}
\newtheorem{question}[theorem]{Question}
\newtheorem{corollary}[theorem]{Corollary}
\newtheorem{proposition}[theorem]{Proposition}
\newtheorem{claim}[name]{Claim}
\theoremstyle{definition}
\newtheorem{definition}[theorem]{Definition}
\newtheorem{example}[theorem]{Example}

\newcommand{\betrag}[1]{\vert{#1}\vert}
\newcommand{\height}[1]{{\rm{ht}}(#1)}
\newcommand{\lub}{{\rm{lub}}}
\newcommand{\cof}[1]{{{\rm{cof}}(#1)}}
\newcommand{\rank}[2]{{\rm{rnk}}_{#2}({#1})}

\newcommand{\map}[3]{{#1}:{#2}\longrightarrow{#3}}
\newcommand{\Map}[5]{{#1}:{#2}\longrightarrow{#3};~{#4}\longmapsto{#5}}
\newcommand{\Set}[2]{\left\{{#1}~\vert~{#2}\right\}}
\newcommand{\BigSet}[2]{\big\{{#1}~\vert~{#2}\big\}}
\newcommand{\seq}[2]{\langle{#1}~\vert~{#2}\rangle}
\newcommand{\eins}{ {1{\rm\hspace{-0.5ex}l}} }
\newcommand{\Add}[2]{{\rm{Add}}({#1},{#2})}
\newcommand{\id}{{\rm{id}}}

\newcommand{\LL}{{\rm{L}}}
\newcommand{\KK}{{\rm{K}}}
\newcommand{\ZFC}{{\rm{ZFC}}}
\newcommand{\DDD}{{\mathbb{D}}}
\newcommand{\III}{{\mathbb{I}}}

\newcommand{\TTT}{{\mathbb{T}}}
\newcommand{\ZZZ}{{\mathbb{Z}}}
\newcommand{\VV}{{\rm{V}}}
\newcommand{\calG}{\mathcal{G}}
\newcommand{\calL}{\mathcal{L}}
\newcommand{\calM}{\mathcal{M}}

\newcommand{\Aut}[1]{{{\rm{Aut}}({#1})}}
\newcommand{\Sym}[1]{{{\rm{Sym}}({#1})}}
\newcommand{\Fin}[1]{{{\rm{Fin}}({#1})}}

\title{Free groups and automorphism groups of infinite fields}

\author{Philipp L\"ucke}
\address{Mathematisches Institut\\Rheinische Friedrich-Wilhelms-Universit\"at Bonn\\En\-de\-nicher Allee 60\\53115 Bonn\\Germany}
\email{pluecke@math.uni-bonn.de}

\author{Saharon Shelah}
\address{Institute of Mathematics, The Hebrew University of Jerusalem, Einstein Institute of Mathematics, Edmond J. Safra Campus, Givat Ram, Jerusalem 91904, Israel, \and Department of Mathematics, Hill Center-Busch Campus, Rutgers, The State University of New Jersey, 110 Frelinghuysen Road, Piscataway, NJ 08854-8019, USA}
\email{shelah@math.huji.ac.il}

\thanks{The second author would like to thank the Israel Science Foundation for partial support of this research (Grant No. 1053/11). This is item \#1014 on Shelah's list of publications.}

\begin{document}

\begin{abstract}
Let $\lambda$ be a cardinal with $\lambda=\lambda^{\aleph_0}$ and $p$ be either $0$ or a prime number. We show that there are fields $K_0$ and $K_1$ of cardinality $\lambda$ 
 and characteristic $p$ such that the automorphism group of $K_0$ is a free group of cardinality $2^\lambda$ and the automorphism group of $K_1$ is a free abelian group of 
cardinality $2^\lambda$. 
 This partially answers a question from \cite{MR1736959} and complements results from \cite{MR1934424}, \cite{MR2773054} and \cite{MR1720580}. 
The methods developed in the proof of the above statement also allow us to show that the above cardinal arithmetic assumption is consistently not necessary for the existence 
of such fields and that 
the existence of a cardinal $\lambda$ of uncountable cofinality with the property that there is no field of cardinality $\lambda$ whose automorphism group 
is a free group of cardinality greater than $\lambda$ implies the existence of large cardinals in certain inner models of set theory.
\end{abstract}
\maketitle


\section{Introduction}\label{section:Introduction}

The work of this paper is motivated by questions of the following type: \emph{given an abstract group $G$ and an infinite cardinal $\lambda$, is $G$ isomorphic to the 
automorphism group of a field of cardinality $\lambda$?} We start by presenting some known results related to this kind of problem.

If $K$ is an infinite field of cardinality $\lambda$, then the group $\Aut{K}$ consisting of all automorphisms of $K$ can be embedded into the group $\Sym{\kappa}$ 
of all permutations of $\lambda$  and therefore has cardinality at most $2^\lambda$. 
It is well known (see \cite{MR660867}, \cite{Sh913} and Section \ref{section:MI} of this paper) that, given an infinite cardinal $\lambda$, 
a first-order language $\calL$ of cardinality at most $\lambda$ and an $\calL$-model $\calM$ of cardinality at most $\lambda$, 
there is a field $K$ of arbitrary characteristic and cardinality $\lambda$ whose automorphism group is isomorphic to $\Aut{\calM}$.
Given an infinite group $G$ of cardinality $\lambda$, it is easy to construct a first-order language $\calL$ of cardinality $\lambda$ and an $\calL$-model $\calM$ 
of cardinality $\lambda$ such that the groups $G$ and $\Aut{\calM}$ are isomorphic. 
In particular, every infinite group is isomorphic to the automorphism group of a field of the same cardinality. 
In contrast, for every infinite cardinal $\lambda$ there are groups of cardinality $\lambda^+$ that are not isomorphic to automorphism groups of fields of cardinality $\lambda$. 
For example, De Bruijn showed in {\cite[Theorem 5.1]{MR0098127}} that the group $\Fin{\lambda^+}$ consisting of all 
finite permutations of $\lambda^+$ cannot be embedded into the group $\Sym{\lambda}$.

In this paper, we focus on free groups and the following instances of the above problem.

\begin{question}\label{question:Motiv}
 Is there a field $K$ whose automorphism group is a free group of cardinality greater than the cardinality of $K$?

More specifically, given an infinite cardinal $\lambda$, is there a field of cardinality $\lambda$ whose automorphism group is a free group of cardinality greater than $\lambda$? 
\end{question}

The above question was first asked by David Evans for the case $\lambda=\aleph_0$. 
The results of \cite{MR1736959} motivate its generalizations to uncountable cardinalities.

The following results due to the second author show that the second part of Question \ref{question:Motiv} 
has a negative answer for $\lambda=\aleph_0$ and singular strong limit cardinals of countable cofinality.

 \begin{theorem}[{\cite[Theorem 1]{MR1934424}}]\label{theorem:ST1}
 Let $\calL$ be a countable first-order language and $\calM$ be a countable $\calL$-model. Then $\Aut{\calM}$ is not an uncountable free group.
\end{theorem}

\begin{theorem}[{\cite[Remark 5.2]{MR1934424}}]\label{theorem:ST2}
 Let $\seq{\lambda_n}{n<\omega}$ be a sequence of infinite cardinals with $2^{\lambda_n}<2^{\lambda{n+1}}$ for all $n<\omega$, $\lambda=\sum_{n<\omega}\lambda_n$ and 
$\mu=\sum_{n<\omega}2^{\lambda_n}$. If $\calL$ is a first-order language of cardinality $\lambda$ and $\calM$ is an $\calL$-model of cardinality $\lambda$ such that 
$\Aut{\calM}$ has cardinality greater than $\mu$, then $\Aut{\calM}$ is not a free group.
\end{theorem}

In contrast, Just, Thomas and the second author showed in {\cite[Theorem 1.14]{MR1736959}} that, given a regular uncountable cardinal $\lambda$ with 
$\lambda=\lambda^{{<}\lambda}$ and $\nu>\lambda$, there is a cofinality preserving forcing extension of the ground model that adds no new sequences of ordinals 
of length less than $\lambda$ and contains a field of cardinality $\lambda$ whose automorphism group is a free group of cardinality $\nu$. 
In particular, it is consistent with the axioms of $\ZFC$ that the above question has a positive answer.

The following main result of this paper shows that the axioms of $\ZFC$ already imply a positive answer to the above question  
for large class of cardinals of uncountable cofinality.

\begin{theorem}\label{theorem:Main1}
 Let $\lambda$ be a cardinal with $\lambda=\lambda^{\aleph_0}$ and $p$ be either $0$ or a prime number. Then there is a field $K$ of characteristic $p$ and cardinality $\lambda$ 
whose automorphism group is a free group of cardinality $2^\lambda$.
\end{theorem}

Since the axioms of $\ZFC$ prove the existence of a cardinal $\lambda$ with $\lambda=\lambda^{\aleph_0}$, this results answers the first part of Question \ref{question:Motiv} positively. 
Moreover, a combination of the above results allows us to completely answer the second part of the question under certain cardinal arithmetic assumptions. 
The following corollary is an example of such an application.

\begin{corollary}\label{corollary:CHSCHAnswer}
 Assume that the \emph{Continuum Hypothesis} and the \emph{Singular Cardinal Hypothesis} hold. Then the following statements are equivalent for every infinite cardinal $\lambda$. 
\begin{enumerate}
 \item There is a field of cardinality $\lambda$ whose automorphism group is a free group of cardinality greater than $\lambda$.

 \item There is a cardinal $\kappa\leq\lambda$ with $2^\kappa>\lambda$ and $\cof{\kappa}>\omega$.
\end{enumerate}
\end{corollary}

We outline the proof of Theorem \ref{theorem:Main1}: In Section \ref{section:MI}, we will show that it suffices to construct an inverse system groups satisfying certain 
cardinality assumptions whose inverse limit is a free group of large cardinality. We will construct such systems of groups assuming the existence of certain inverse systems of 
sets in Section \ref{section:freegroups}. Finally, we will use the assumption $\lambda=\lambda^{\aleph_0}$ to construct suitable inverse systems of sets 
in Section \ref{section:system}.

The methods developed in the proof of the above result also allow us to produce uncountable fields whose automorphism group is a free \emph{abelian} group of large cardinality.

\begin{theorem}\label{theorem:Main2}
 Let $\lambda$ be a cardinal with $\lambda=\lambda^{\aleph_0}$ and $p$ be either $0$ or a prime number. Then there is a field $K$ of characteristic $p$ and cardinality 
$\lambda$ whose automorphism group is a free abelian group of cardinality $2^\lambda$.
\end{theorem}

Again, this drastically contrasts the countable setting as the following result due to S\l awomir Solecki shows.

\begin{theorem}[{\cite[Remark 1.6]{MR1720580}}]
 Let $\calL$ be a countable first-order language and $\calM$ be an $\calL$-model. Then $\Aut{\calM}$ is not an uncountable free abelian group.
\end{theorem}

In another direction, the methods developed in the proofs of the above results also allow us to show that the cardinal arithmetic assumption $\lambda=\lambda^{\aleph_0}$ 
is consistently not necessary for the existence of a field of cardinality $\lambda$ whose automorphism group is a free group of cardinality greater than $\lambda$. 
This is an implication of the following result. Given a cardinal $\lambda$, we use $\Add{\omega}{\lambda}$ to denote the forcing that adds $\lambda$-many Cohen reals to the ground model.

\begin{theorem}\label{theorem:OuterModelFields}
 Let $\lambda$ be a cardinal with $\lambda=\lambda^{\aleph_0}$ and $p$ be either $0$ or a prime number. 
If $G$ is $\Add{\omega}{\kappa}$-generic over the ground model $\VV$ for some cardinal $\kappa$, then there is a field $K$ of characteristic $p$ and cardinality $\lambda$ 
contained in $\VV[G]$ whose automorphism group is a free group of cardinality greater than or equal to $(2^\lambda)^\VV$ in $\VV[G]$. 
\end{theorem}

The above results raise the question whether the existence of a cardinal $\lambda$ of uncountable cofinality with the property that there is no field of cardinality $\lambda$ 
whose automorphism group is a free group of cardinality greater than $\lambda$ is even consistent with the axioms of $\ZFC$. 
Another byproduct of our constructions is the observation that the existence of such a cardinal has consistency strength strictly greater than that of $\ZFC$.
This observation is a consequence of the next result.

Remember that a partial order $\TTT=\langle T,\leq_\TTT\rangle$ is a \emph{tree} if 
$\TTT$ has a unique minimal element and the set $prec_\TTT(t)=\Set{s\in T}{s\leq_\TTT t,~s\neq t}$ is a well-ordered by $\leq_\TTT$ for every $t\in T$. 
Given such a tree $\TTT$ and $t\in T$, we define $\rank{t}{\TTT}$ to be the order-type of $\langle prec_\TTT(t),\leq_\TTT\rangle$. We call the ordinal 
$\height{\TTT}=\lub\Set{\rank{t}{\TTT}}{t\in T}$ the \emph{height} of $\TTT$. Finally, a subset $B$ of $T$ is a \emph{cofinal branch through $\TTT$} 
if $B$ is $\leq_\TTT$-downwards closed and $B$ is well-ordered by $\leq_\TTT$ with order-type $\height{\TTT}$.

\begin{theorem}\label{theorem:InnerModelFields}
 Let $\lambda$ be a cardinal of uncountable cofinality. 
 If there is a tree of cardinality and height $\lambda$ with more than $\lambda$-many cofinal branches, then there is a field of cardinality $\lambda$ 
 whose automorphism group is a free group of cardinality bigger than $\lambda$.
\end{theorem}

By considering the tree $\langle({}^{{<}\lambda}2)^M,\subseteq\rangle$ for some inner model $M$, this result directly implies the following corollary.

\begin{corollary}
 Let $\lambda$ be a cardinal of uncountable cofinality and $M$ be an inner model $\ZFC$ with $(\lambda^+)^M=\lambda^+$. 
 If $\lambda=(\lambda^{{<}\lambda})^M$, then there is a field of cardinality $\lambda$ 
 whose automorphism group is a free group of cardinality bigger than $\lambda$. \qed
\end{corollary}

This statement directly allows us to derive large cardinal strength from the non-existence of certain fields.

\begin{corollary}
 Let $\lambda$ be a regular uncountable cardinal such that there is no field of cardinality $\lambda$ whose automorphism group is a free group of cardinality greater than 
$\lambda$. Then $\lambda^+$ is an inaccessible cardinal in $\LL[x]$ for every $x\subseteq\lambda$.
\end{corollary}

\begin{proof}
Assume, towards a contradiction, that $\lambda^+$ is not an inaccessible cardinal in $\LL[x]$ for some $x\subseteq\lambda$. Then there is a $y\subseteq\lambda$ with 
$\lambda^+=(\lambda^+)^{\LL[y]}$ and $\langle({}^{{<}\kappa}2)^{\LL[y]},\subseteq\rangle$ is a tree of cardinality and height $\lambda$, because our assumptions imply 
$(\lambda^{{<}\lambda})^{\LL[y]}=\lambda$. But the set of cofinal branches through this tree has cardinality at least $(2^\lambda)^{\LL[y]}=(\lambda^+)^{\LL[y]}=\lambda^+$, 
a contradiction. 
\end{proof}

Note that Mitchell used an inaccessible cardinal to constructed a model of $\ZFC$ in which every tree of cardinality and height $\omega_1$ has at most $\aleph_1$-many cofinal 
branches (see {\cite[Section 8]{MR823775}} and \cite{MR0313057}). This statement is also a consequence of the \emph{Proper Forcing Axiom} (see {\cite[Theorem 7.10]{MR776640}}).

In the case of singular cardinals of uncountable cofinality, it is possible to use \emph{core model theory} 
(see, for example, \cite{MR2768699}) to obtain inner models containing much larger large cardinals from the above assumption.

\begin{corollary}
  Let $\lambda$ be a singular cardinal of uncountable cofinality such that there is no field of cardinality $\lambda$ whose automorphism group is a free group of cardinality 
greater than $\lambda$. Then there is an inner model with a Woodin cardinal.
\end{corollary}

\begin{proof}
Assume, towards a contradiction, that there is no inner model with a Woodin cardinal. Then we can construct the \emph{core model $\KK$ below one Woodin cardinal}. 
It satisfies the \emph{Generalized Continuum Hypothesis} and has the \emph{covering property}. In particular, we have 
$\lambda^+=(\lambda^+)^{\KK}$. But this means that $\langle({}^{{<}\lambda}2)^{\KK},\subseteq\rangle$ is a tree of cardinality and height $\lambda$ and the set of cofinal 
branches through this tree has cardinality at least $(2^\lambda)^{\KK}=(\lambda^+)^{\KK}=\lambda^+$, a contradiction.
\end{proof}

The results of {\cite[Section 2]{MR1812172}} show that the non-existence of such trees at a singular cardinal of uncountable cofinality is equivalent to a 
\rm{PCF}-theoretic statement that is not known to be consistent. Related questions can also be found in {\cite[Chapter II, Section 6]{MR1318912}}.


\section{Representing inverse limits as automorphism groups}\label{section:MI}

In this section, we start from an inverse system of groups $\III$ to construct a first-order language $\calL$ and an $\calL$-model $\calM$ such that $\Aut{M}$ 
is an inverse limit of $\III$ and the cardinalities of $\calL$ and $\calM$ only depend on the cardinalities of the groups in $\III$ and the cardinality of the 
underlying directed set. 
We start by recalling some standard definitions and presenting the relevant examples.

We call a pair $\DDD=\langle D,\leq_\DDD\rangle$ a \emph{directed set} if $\leq_\DDD$ is a reflexive, transitive binary relation on the set $D$ 
 with the property that for all $p,q\in D$ there is a $r\in D$ with $p,q\leq_\DDD r$.


Given a directed set $\DDD=\langle D,\leq_\DDD\rangle$, we call a pair 
\begin{equation*}
 \III=\langle\seq{A_p}{p\in D},\seq{f_{p,q}}{p,q\in D,~p\leq_\DDD q}\rangle
\end{equation*}
an \emph{inverse system of sets over $\DDD$} if the following statements hold for all $p,q,r \in D$ with $p\leq_\DDD q\leq_\DDD r$.
\begin{enumerate}
 \item $A_p$ is a non-empty set and $\map{f_{p,q}}{A_q}{A_p}$ is a function.

 \item $f_{p,p}=\id_{A_p}$ and $f_{p,q}\circ f_{q,r}=f_{p,r}$.
\end{enumerate}
Given such an inverse system $\III$, we call the set 
\begin{equation*}
  A_\III = \BigSet{(a_p)_{p\in D}}{\textit{$f_{p,q}(a_q)=a_p$ for all $p,q\in D$ with $p\leq_\DDD q$}}
\end{equation*}
the \emph{inverse limit of $\III$}.

\begin{example}\label{example:LambdaCountableInverseLimit}
 Let $\lambda$ be an infinite cardinal and let $[\lambda]^{\aleph_0}$ denote the set of all countable subsets of $\lambda$. 
 Given $u,v\in[\lambda]^{\aleph_0}$ with $u\subseteq v$, 
 set $A_u={}^u 2$ and define $\map{f_{u,v}}{A_v}{A_u}$ by $f_{u,v}(s)=s\restriction u$ for all $s\in{}^v 2$. 
 Let 
 \begin{equation*}
  \III_\lambda ~ = ~ \langle\seq{A_u}{u\in[\lambda]^{\aleph_0}},\seq{f_{u,v}}{u,v\in[\lambda]^{\aleph_0},~u\subseteq v}\rangle 
 \end{equation*}
 denote the resulting inverse system of sets over the directed set $\langle [\lambda]^{\aleph_0},\subseteq\rangle$. 

 Then it is easy to see that
 \begin{equation*}
  \Map{b}{{}^\lambda 2}{A_{\III_\lambda}}{x}{(x\restriction u)_{u\in[\lambda]^{\aleph_0}}}
 \end{equation*}
 is a well-defined bijection between the sets ${}^\lambda 2$ and $A_{\III_\lambda}$.
\end{example}

\begin{example}\label{example:TreesInverseSystems}
 Let $\TTT=\langle T,\leq_\TTT\rangle$ be a tree. Given $\alpha<\height{\TTT}$, we let $\TTT(\alpha)$ denote the set of all $t\in T$ with $\rank{t}{\TTT}=\alpha$. 
 If $t\in T$ and $\alpha\leq\rank{t}{\TTT}$, then we let $t\restriction\alpha$ denote the unique element $s\in \{t\}\cup prec_\TTT(t)$ with $\rank{s}{\TTT}=\alpha$.

  Given $\alpha\leq\beta<\height{\TTT}$, set $A_\alpha=\TTT(\alpha)$ and 
 \begin{equation*}
  \Map{f_{\alpha,\beta}}{A_\beta}{A_\alpha}{t}{t\restriction\alpha}.
 \end{equation*} 
 We let 
 \begin{equation*}
  \III_\TTT ~ = ~ \langle\seq{A_\alpha}{\alpha<\height{\TTT}},\seq{f_{\alpha,\beta}}{\alpha\leq\beta<\height{\TTT}}\rangle
 \end{equation*}
 denote the resulting inverse system of sets over the directed set $\langle\height{\TTT},\leq\rangle$.

 It is easy to see that the induced map
 \begin{equation*}
  \Map{b}{A_{\III_\TTT}}{[\TTT]}{(a_\alpha)_{\alpha<\lambda}}{\Set{a_\alpha}{\alpha<\lambda}}
 \end{equation*}
 is a bijection between the inverse limit $A_{\III_\TTT}$ and the set $[\TTT]$ consisting of all cofinal branches through $\TTT$.
\end{example}


We now consider inverse limits in the category of groups. 
Given a directed set $\DDD=\langle D,\leq_\DDD\rangle$, a pair
\begin{equation*}
 \III=\langle\seq{G_p}{p\in D},\seq{h_{p,q}}{p,q\in D,~p\leq_\DDD q}\rangle 
\end{equation*}
is an \emph{inverse system of groups over $\DDD$} if the following statements hold.
\begin{enumerate}
 \item $G_p$ is a group with underlying set $X_p$ for all $p\in D$.

 \item If $p,q\in D$ with $p\leq_\DDD q$, then $\map{h_{p,q}}{G_q}{G_p}$ is a homomorphism of groups.

 \item The pair
   \begin{equation*}
    \langle\seq{X_p}{p\in D},\seq{h_{p,q}}{p,q\in D,~p\leq_\DDD q}\rangle
   \end{equation*}
  is an inverse system of sets.
\end{enumerate}

Given such a system $\III$, we call the subgroup
\begin{equation*}
 G_{\III}=\BigSet{(g_p)_{p\in D}\in\prod_{p\in D}G_p}{\textit{$h_{p,q}(g_q)=g_p$ for all $p,q\in D$ with $p\leq_\DDD q$}}
\end{equation*}
of the direct product of the $G_p$'s the \emph{inverse limit of $\III$}.

\begin{example}\label{example:LimitOfFreeGroups}
 Let $\DDD=\langle D,\leq_\DDD\rangle$ be a directed set and 
 \begin{equation*}
  \III=\langle\seq{A_p}{p\in D},\seq{f_{p,q}}{p,q\in D,~p\leq_\DDD q}\rangle  
 \end{equation*}
 be an inverse system of sets over $\DDD$. For each $p\in D$, let $G_p$ be the free group with generators $\Set{x_{p,a}}{a\in A_p}$. 
 Given $p,q\in D$ with $p\leq_\DDD q$, we let $\map{h_{p,q}}{G_q}{G_p}$ denote the unique homomorphism of groups with 
 $h_{p,q}(x_{q,a})=x_{p,f_{p,q}(a)}$ for all $a\in A_q$. Let 
 \begin{equation*}
  \III_{gr} ~ = ~ \langle\seq{G_p}{p\in D},\seq{h_{p,q}}{p,q\in D,~p\leq_\DDD q}\rangle
 \end{equation*}
 denote the resulting inverse system of groups over $\DDD$.
\end{example}

It is an obvious question whether the corresponding inverse limit $G_{\III_{gr}}$ is itself a free group. In Section \ref{section:freegroups} 
we will present conditions that imply this statement. These implications will allow us to prove Theorem \ref{theorem:Main1}.

\begin{example}
 Pick $\DDD$ and $\III$ as in Example \ref{example:LimitOfFreeGroups}. 
 For each $p\in D$, let $H_p$ be the free abelian group with basis $\Set{x_{p,a}}{a\in A_p}$. 
 Define $\map{h_{p,q}}{H_q}{H_p}$ for all $p,q\in D$ with $p\leq_\DDD q$ as above and let 
 \begin{equation*}
  \III_{ab} ~ = ~ \langle\seq{H_p}{p\in D},\seq{h_{p,q}}{p,q\in D,~p\leq_\DDD q}\rangle
 \end{equation*}
 denote the resulting inverse system of groups over $\DDD$. Since every $H_p$ is an abelian group, the group $G_{\III_{ab}}$ is also abelian.
\end{example}


This section focuses on the proof of the following result.

\begin{theorem}\label{theorem:ModelIAut}
 Let $\III=\langle\seq{G_q}{q\in D},\seq{h_{q,r}}{q,r\in D,~q\leq_\DDD r}\rangle$ 
be an inverse system of groups over a directed set $\DDD=\langle D,\leq_\DDD\rangle$ and $p$ be either $0$ or a prime number. Then there is a field $K$ of characteristic $p$ with the following properties.
\begin{enumerate}
 \item The groups $\Aut{K}$ and $G_{\III}$ are isomorphic.
 \item $\betrag{K} \leq \max\{\aleph_0,\sum_{q\in D}\betrag{G_q}\}$.
\end{enumerate}
\end{theorem}

In the remainder of this section, we fix $\DDD$ and $\III$ as in the statement of the theorem. 
We define $\lambda$ to be the cardinal $\max\{\aleph_0,\sum_{q\in D}\betrag{G_q}\}$.

The following results show that it suffices to find a first-order language $\calL_\III$ of cardinality $\lambda$ and an $\calL_\III$-model $\calM_\III$ of cardinality $\lambda$ 
such that the groups $\Aut{\calM_\III}$ and $G_{\III}$ are isomorphic.

\begin{proposition}
 Let $\lambda$ be an infinite cardinal and $\calL$ be a first order language of cardinality at most $\lambda$. 
 If $\calM$ is an $\calL$-model of cardinality at most $\lambda$, then there is a connected graph $\Gamma=\langle X,E\rangle$ such that $\betrag{X}=\lambda$ 
and the groups $\Aut{\calM}$ and $\Aut{\Gamma}$ are isomorphic.
\end{proposition}

\begin{proof}
 The above statement can be derived from the results of {\cite[Section 5.5]{MR1221741}} and {\cite[Section 3]{Sh913}}.
\end{proof}

\begin{theorem}[\cite{MR660867} and {\cite[Main Theorem B]{Sh913}}]
 Let $\Gamma=\langle X,E\rangle$ be a connected graph and $p$ be either $0$ or a prime number. Then there is a field $K$ of characteristic $p$ with the following properties.
\begin{enumerate}
 \item The groups $\Aut{\Gamma}$ and $\Aut{K}$ are isomorphic.
 \item $\betrag{K}\leq\max\{\aleph_0,\betrag{X}\}$.
\end{enumerate}
\end{theorem}

We are now ready to construct $\calL_\III$ and $\calM_\III$ with the properties stated above.

Define $\calL_\III$ to be first-order language with the following symbols.
\begin{itemize}
 \item Constant symbols $\dot{c}_{g,q}$ for all $q\in D$ and $g\in G_q$.

 \item Unary relation symbols $\dot{P}_q$ for all $q\in D$.

 \item Binary relation symbols $\dot{H}_{q,r}$ for all $q,r\in D$ with $q\leq_\DDD r$.

 \item Ternary relations symbols $\dot{F}_q$ for all $q\in D$.
\end{itemize}

Let $\calM_\III$ denote the unique $\calL_\III$-model with the following properties.
\begin{itemize}
 \item The domain of $\calM_\III$ is the set
  \begin{equation*}
   M_\III=\Set{\langle g,q,i\rangle}{q\in D,~g\in G_q,~i<2}.
  \end{equation*}

 \item $\dot{c}_{g,q}^{\calM_\III}=\langle g,q,1\rangle$ for all $q\in D$ and $g\in G_q$.

 \item $\dot{P}_q^{\calM_\III}=\Set{\langle g,q,0\rangle}{g\in G_q}$ for all $q\in D$.

 \item $\dot{H}_{q,r}^{\calM_\III}=\Set{\langle\langle g,r,0\rangle,\langle h_{q,r}(g),q,0\rangle\rangle}{g\in G_r}$ for all $q,r\in D$ with $q\leq_\DDD r$.

 \item $\dot{F}_q^{\calM_\III}=\Set{\langle\langle g,q,0\rangle,\langle h,q,1\rangle,\langle g\cdot h,q,0\rangle\rangle}{g,h\in G_q}$ for all $q\in D$.
\end{itemize}

\begin{proposition}\label{proposition:SigmaRestrictionQ}
 If $\sigma\in\Aut{\calM_\III}$, $q\in D$ and $g\in G_q$, then $\sigma(\langle g,q,1\rangle)=\langle g,q,1\rangle$ and $\map{\sigma\restriction\dot{P}_q^{\calM_\III}}{\dot{P}_q^{\calM_\III}}{\dot{P}_q^{\calM_\III}}$. \qed
\end{proposition}

\begin{proposition}\label{proposition:proposition:SigmaRestrictionP}
 If $\sigma\in\Aut{\calM_\III}$ and $q\in D$, then there is a unique $c_{\sigma,q}\in G_q$ with $\sigma(\langle g,q,0\rangle)=\langle c_{\sigma,q}\cdot g,q,0\rangle$ 
for all $g\in G_q$.
\end{proposition}

\begin{proof}
 By Proposition \ref{proposition:SigmaRestrictionQ}, there is a unique $c_{\sigma,q}\in G_q$ such that $\sigma(\langle \eins_{G_q},q,0\rangle)=\langle c_{\sigma,q},q,0\rangle$. 
Given $g\in G_q$, we have $\dot{F}_q^{\calM_\III}(\langle c_{\sigma,q},q,0\rangle,\langle g,q,1\rangle,\sigma(\langle g,q,0\rangle)\rangle)$ 
and this implies $\sigma(\langle g,q,0\rangle)=\langle c_{\sigma,q}\cdot g,q,0\rangle$.
\end{proof}

\begin{proposition}
 If $\sigma\in\Aut{\calM_\III}$ and $q,r\in D$ with $q\leq_\DDD r$, then $h_{q,r}(c_{\sigma,r})=c_{\sigma,q}$. 
In particular, the sequence $(c_{\sigma,q})_{q\in D}$ is an element of $G_\III$ for every $\sigma\in\Aut{\calM_\III}$.
\end{proposition}

\begin{proof}
 The definition of $\calM_\III$ yields $\dot{H}_{q,r}^{\calM_\III}(\langle\eins_{G_r},r,0\rangle,\langle\eins_{G_q},q,0\rangle)$. We can conclude 
$\dot{H}_{q,r}^{\calM_\III}(\langle c_{\sigma,r},r,0\rangle,\langle c_{\sigma,q},q,0\rangle)$ and $h_{q,r}(c_{\sigma,r})=c_{\sigma,q}$.
\end{proof}

\begin{lemma}\label{lemma:IsomorphismModelIInverseLimit}
The map
\begin{equation*}
 \Map{\Phi}{\Aut{\calM_\III}}{G_\III}{\sigma}{(c_{\sigma_q})_{p\in D}}
\end{equation*}
is an isomorphism of groups.
\end{lemma}

\begin{proof}
 Given $\sigma_0,\sigma_1\in\Aut{\calM_\III}$ and $q\in D$, we have
 \begin{equation*}
    (\sigma_1\circ\sigma_0)(\langle \eins_{G_q},q,0\rangle) = \sigma_1 (\langle c_{\sigma_0,q},q,0\rangle) = \langle c_{\sigma_1,q}\cdot c_{\sigma_0,q},q,0\rangle
 \end{equation*}
 and therefore $c_{\sigma_1\circ\sigma_0,q}=c_{\sigma_1,q}\cdot c_{\sigma_0,q}$. This shows that $\Phi$ is a homomorphism.

 Given $\vec{g}=(g_q)_{q\in D}\in G_\III$, we define $\map{\sigma_{\vec{g}}}{M_\III}{M_\III}$ by the following clauses.
 \begin{enumerate}
  \item $\sigma_{\vec{g}}(\langle g,q,0\rangle)=\langle g_q\cdot g,q,0\rangle$ for all $q\in D$ and $g\in G_q$.

  \item $\sigma_{\vec{g}}(\langle g,q,1\rangle)=\langle g,q,1\rangle$ for all $q\in D$ and $g\in G_q$.
 \end{enumerate}
 Then $\sigma_{\vec{g}}\in\Aut{\calM_\III}$ and $c_{\sigma_{\vec{g}},q}=g_q$ for all $q\in D$. This shows that $\Phi$ is surjective. 
Since Propositions \ref{proposition:SigmaRestrictionQ} and \ref{proposition:proposition:SigmaRestrictionP}  
imply that $\Phi$ is also injective, this concludes the proof of the lemma.
\end{proof}

\begin{proof}[Proof of Theorem \ref{theorem:ModelIAut}]
By our assumptions, both $\calL_\III$ and $\calM_\III$ have cardinality at most $\lambda$.  
By the results mentioned above, there is a field $K$ of characteristic $p$ and cardinality $\lambda$ such that the groups 
$\Aut{\calM_\III}$ and $\Aut{K}$ are isomorphic. By Lemma \ref{lemma:IsomorphismModelIInverseLimit}, this completes the proof of the theorem.
\end{proof}


\section{Representing free groups as inverse limits}\label{section:freegroups}

This sections shows how free groups can be represented as inverse limits of systems of groups assuming the existence of certain \emph{suitable} inverse systems of sets. 
In the following, we prove statements from assumptions much weaker than the ones present in Theorem \ref{theorem:Main1} to motivate possible strengthenings of this result.

Let $\DDD=\langle D,\leq_\DDD\rangle$ be a directed set. We define an infinite game $\calG(\DDD)$ of perfect information between Player I and Player II: 
in the $i$-th round of this game Player I chooses an element $p_{2i}$ from $D$ and then Player II chooses an element $p_{2i+1}$ from $D$. 
Player I wins a run $(p_i)_{i<\omega}$ of $\calG(\DDD)$ if and only if either there is an $i<\omega$ with $p_{2i}\not\leq_\DDD p_{2i+1}$ or $p_{2i+1}\leq_\DDD p_{2i+2}$ 
holds for all $i<\omega$ and there is a $p\in D$ with $p_i\leq_\DDD p$ for all $i<\omega$.\footnote{A similar game can be used to characterize the $\sigma$-distributivity of Boolean algebras. See {\cite{MR739910}}.} 
A \emph{winning strategy} for Player II is a function $\map{s}{{}^{{<}\omega}D}{D}$ with the property that Player II wins every run $(p_i)_{i<\omega}$ that is played 
according to $s$, in the sense that $s(\langle p_0,\dots,p_{2i}\rangle)=p_{2i+1}$ holds for all $i<\omega$.

\begin{theorem}\label{theorem:FreeGroupFromSystem}
 Let $\DDD=\langle D,\leq_\DDD\rangle$ be a directed set with the property that Player II has no winning strategy in $\calG(\DDD)$. 
 If $\III$ is an inverse system of sets over $\DDD$ with $A_\III\neq\emptyset$, 
 then the inverse limit $G_{\III_{gr}}$ is a free group of cardinality $\max\{\aleph_0,\betrag{A_\III}\}$.
\end{theorem}

\begin{proof}
Let $\III=\langle\seq{A_p}{p\in D},\seq{f_{p,q}}{p,q\in D,~p\leq_\DDD q}\rangle$. 
If $\vec{g}=(g_p)_{p\in D}\in G_{\III_{gr}}$ and $p\in D$, then we let
\begin{itemize}
 \item $n(\vec{g},p)<\omega$,

 \item $k(\vec{g},p,1), ~ \dots ~ ,k(\vec{g},p,n(\vec{g},p))\in\ZZZ\setminus\{0\}$,

 \item $a(\vec{g},p,1), ~ \dots ~ ,a(\vec{g},p,n(\vec{g},p))\in A_p$
\end{itemize}
 denote the uniquely determined objects with the property that the word
\begin{equation*}
 w_{\vec{g},p} ~ = ~ x_{p,a(\vec{g},p,1)}^{k(\vec{g},p,1)} ~ \dots ~ x_{p,a(\vec{g},p,n(\vec{g},p))}^{k(\vec{g},p,n(\vec{g},p))}
\end{equation*}
is the unique reduced word representing $g_p$, i.e. $w_{\vec{g},p}$ represents $g_p$ and
\begin{equation*}
 a(\vec{g},p,i)\neq a(\vec{g},p,i+1) 
\end{equation*}
for all $1\leq i<n(\vec{g},p)$ 
(see {\cite[2.1.2]{MR1357169}}).

\begin{claim}\label{claim:One}
 If $\vec{g}=(g_p)_{p\in D}\in G_{\III_{gr}}$ and $p,q\in D$ with $p\leq_\DDD q$, then $n(\vec{g},p)\leq n(\vec{g},q)$.
\end{claim}

\begin{proof}[Proof of the Claim]
 Since $h_{p,q}(g_q)=g_p$, we know that the word
 \begin{equation}\label{equation:AlterWord}
  w ~ = ~ x_{p,f_{p,q}(a(\vec{g},q,1))}^{k(\vec{g},q,1)} ~ \dots ~ x_{p,f_{p,q}(a(\vec{g},q,n(\vec{g},q)))}^{k(\vec{g},q,n(\vec{g},q))}
 \end{equation}
also represents $g_p$. Hence $w_{\vec{g},p}$ can be obtained from $w$ by a finite number of reductions. This implies $n(\vec{g},p)\leq n(\vec{g},q)$.
\end{proof}

\begin{claim}
 If $\vec{g}\in G_{\III_{gr}}$, then there are $p_{\vec{g}}\in D$ and $n_{\vec{g}}<\omega$ such that $n_{\vec{g}}=n(\vec{g},p)$ for all $p\in D$ with $p_{\vec{g}}\leq_\DDD p$.
\end{claim}

\begin{proof}[Proof of the Claim]
Let $\vec{g}=(g_p)_{p\in D}$ and assume, toward a contradiction, that for every $p\in D$ there is a $q\in D$ with $p\leq_\DDD q$ and $n(\vec{g},p)<n(\vec{g},q)$. Then there is a function $\map{s}{{}^{{<}\omega}D}{D}$ such that $p_{2i}\leq_\DDD s(\langle p_0,\dots,p_{2i}\rangle)$ and $n(\vec{g},p_{2i})<n(\vec{g},s(\langle p_0,\dots,p_{2i}\rangle))$ for all $i<\omega$ and $p_0,\dots,p_{2i}\in D$. By our assumption, $s$ is not a winning strategy for Player II and there is a run $(p_i)_{i<\omega}$ of $\calG(\DDD)$ played according to $s$ that is won by Player I. This gives us a $p\in D$ with $p_i\leq_\DDD p$ for all $i<\omega$. By Claim \ref{claim:One}, we have $n(\vec{g},p)>i$ for all $i<\omega$, a contradiction.
\end{proof}

\begin{claim}\label{claim:Three}
 If $\vec{g}=(g_p)_{p\in D}\in G_{\III_{gr}}$ and $p,q\in D$ with $p_{\vec{g}}\leq_\DDD p\leq_\DDD q$, then $a(\vec{g},p,i)=f_{p,q}(a(\vec{g},q,i))$ and 
$k(\vec{g},p,i)=k(\vec{g},q,i)$ for all $1\leq i<n_{\vec{g}}$.
\end{claim}

\begin{proof}[Proof of the Claim]
 Let $w$ be the word defined in (\ref{equation:AlterWord}). Then $w$ is reduced, because otherwise there would be a reduced word 
$x^{k_0}_{p,a_1}\dots x^{k_{l-1}}_{p,a_l}$ with $l<n(\vec{g},q)=n(\vec{g},p)$ representing $g_p$ 
and this would contradict the choice of $w_{\vec{g},p}$. We can conclude $w=w_{\vec{g},p}$ and, again by the uniqueness of $w_{\vec{g},p}$, this yields the statements of the claim.
\end{proof}

Given $\vec{a}=(a_p)_{p\in D}\in A_\III$, we define 
\begin{equation*}
 \vec{g}_{\vec{a}} ~ = ~ (x_{p,a_p})_{p\in D} ~ \in ~ \prod_{p\in D}G_p.  
\end{equation*}
It is easy to see that $\vec{g}_{\vec{a}}$ is an element of $G_{\III_{gr}}$.

\begin{claim}
 The group $G_{\III_{gr}}$ is generated by the set $\Set{\vec{g}_{\vec{a}}}{\vec{a}\in A_\III}$.
\end{claim}

\begin{proof}[Proof of the Claim]
Let $\vec{g}=(g_p)_{p\in D}\in G_{\III_{gr}}$. Set $n=n_{\vec{g}}$ and $k_i=k(\vec{g},p_{\vec{g}},i)$ for all $1\leq i\leq n$. 
For each $p\in D$, we fix an element $\bar{p}$ of $D$ with $p,p_{\vec{g}}\leq_\DDD \bar{p}$. Given $p\in D$ and $1\leq i\leq n$, 
define $a_{\vec{g},p,i}=f_{p,\bar{p}}(a(\vec{g},\bar{p},i)) \in A_p$.

Let $p,q\in D$ with $p\leq_\DDD q$. Fix an $r\in D$ with $\bar{p},\bar{q}\leq_\DDD r$. 
By Claim \ref{claim:Three}, we have $a(\vec{g},\bar{p},i)=f_{\bar{p},r}(a(\vec{g},r,i))$ and $a(\vec{g},\bar{q},i)=f_{\bar{q},r}(a(\vec{g},r,i))$. This implies
\begin{equation*}
 f_{p,q}(a_{\vec{g},q,i}) ~ = ~ f_{p,q}(f_{q,\bar{q}}(a(\vec{g},\bar{q},i))) ~ = ~ f_{p,r}(a(\vec{g},r,i)) ~ = ~ f_{p,\bar{p}}(a(\vec{g},\bar{p},i)) ~ = ~ a_{\vec{g},p,i}
\end{equation*}
and we can conclude 
\begin{equation*}
 \vec{a}_{\vec{g},i} ~ = ~ (a_{\vec{g},p,i})_{p\in D} ~ \in ~ A_\III.
\end{equation*}
By the above computations, we know that 
\begin{equation*}
 g_p ~ = ~ h_{p,\bar{p}}(g_{\bar{p}}) ~ = ~ h_{p,\bar{p}}\big(x^{k_1}_{\bar{p},a(\vec{g},\bar{p},1)}\cdot ~ \dots ~ \cdot x^{k_n}_{\bar{p},a(\vec{g},\bar{p},n)}\big) ~ = ~ x^{k_1}_{p,a_{\vec{g},p,1}} \cdot ~ \dots ~ \cdot x^{k_n}_{p,a_{\vec{g},p,n}}
\end{equation*}
holds for all $p\in D$ and this shows
\begin{equation*}
 \vec{g} ~ = ~ \vec{g}^{~ k_1}_{\vec{a}_{\vec{g},1}} ~ \cdot ~  \dots ~ \cdot ~ \vec{g}^{~ k_n}_{\vec{a}_{\vec{g},n}}. 
\end{equation*}
\end{proof}

\begin{claim}
 The group $G_{\III_{gr}}$ is freely generated by the set $\Set{\vec{g}_{\vec{a}}}{\vec{a}\in A_\III}$.
\end{claim}

\begin{proof}[Proof of the Claim]
  Assume, toward a contradiction, that we can find $1\leq n<\omega$, $\vec{a}_1,\dots,\vec{a}_n\in A_\III$ and $k_1,\dots,k_n\in\ZZZ\setminus\{0\}$ with 
\begin{equation*}
 \vec{g}_{\vec{a}_1}^{~k_1} ~ \cdot ~ \dots ~ \cdot ~ \vec{g}_{\vec{a}_n}^{~k_n} ~ = ~ \eins_{G_{\III_{gr}}} 
\end{equation*}
and $\vec{a}_i\neq \vec{a}_{i+1}$ for all $1\leq i<n$. Let $\vec{a}_i=(a_{p,i})_{p\in D}$. 
Then there are $p_1,\dots,p_{n-1}\in D$ with $a_{p_i,i}\neq a_{p_i,i+1}$ for all $1\leq i<n$ and we can find a $p\in D$ with $p_1,\dots,p_{n-1}\leq_\DDD p$. 
This means $a_{p,i}\neq a_{p,i+1}$, because otherwise 
\begin{equation*}
 a_{p_i,i} ~ = ~ h_{p_i,p}(a_{p,i}) ~ = ~ h_{p_i,p}(a_{p,i+1}) ~ = ~ a_{p_i,i+1}. 
\end{equation*}
By our assumption, the word $w = x^{k_1}_{p,a_1}\cdot ~ \dots ~ \cdot x^{k_n}_{p,a_n}$ is equivalent to the trivial word. 
But this yields a contradiction, because $w$ is reduced and not trivial.
\end{proof}

 This completes the proof of the theorem.
\end{proof}

A small modification of the above proof yields the corresponding result for free abelian group. 
As usual, we let $[A]^{{<}\aleph_0}$ denote the set of all finite subsets of a given set $A$.

\setcounter{name}{0}

\begin{theorem}\label{theorem:AbelianFreeGroupFromSystem}
 Let $\DDD=\langle D,\leq_\DDD\rangle$ be a directed set with the property that Player II has no winning strategy in $\calG(\DDD)$. 
 If $\III$ is an inverse system of sets over $\DDD$ with $A_\III\neq\emptyset$, 
 then the inverse limit $G_{\III_{ab}}$ is a free abelian group of cardinality $\max\{\aleph_0,\betrag{A_\III}\}$.
\end{theorem}

\begin{proof}
Let $\III=\langle\seq{A_p}{p\in D},\seq{f_{p,q}}{p,q\in D,~p\leq_\DDD q}\rangle$. 
If $\vec{g}=(g_p)_{p\in D}\in G_{\III_{ab}}$, then we let
\begin{itemize}
 \item $I_{\vec{g},p}\in[A_p]^{{<}\aleph_0}$,

 \item $\seq{k(\vec{g},p,a)\in\ZZZ\setminus\{0\}}{a\in I_{\vec{g},p}}$
\end{itemize}
 denote the uniquely determined objects such that
\begin{equation*}
 g_p ~ = ~ \sum_{a\in I_{\vec{g},p}}~k(\vec{g},p,a)\cdot x_{p,a} 
\end{equation*}
holds for all $p\in D$. The following claims can be derived in the same way as the corresponding claims in the proof of Theorem \ref{theorem:FreeGroupFromSystem}.

\begin{claim}
 If $\vec{g}=(g_p)_{p\in D}\in G_{\III_{ab}}$ and $p,q\in D$ with $p\leq_\DDD q$, then $\betrag{I_{\vec{g},p}}\leq\betrag{I_{\vec{g},q}}$. \qed
\end{claim}

\begin{claim}
 If $\vec{g}\in G_{\III_{ab}}$, then there are $p_{\vec{g}}\in D$ and $n_{\vec{g}}<\omega$, such that $n_{\vec{g}}=\betrag{I_{\vec{g},p}}$ 
for all $p\in D$ with $p_{\vec{g}}\leq_\DDD p$. \qed
\end{claim}

\begin{claim}
 If $\vec{g}=(g_p)_{p\in D}\in G_{\III_{ab}}$ and $p,q\in D$ with $p_{\vec{g}}\leq_\DDD p\leq_\DDD q$, then $f_{p,q}\restriction I_{\vec{g},q}$ is a bijection of 
$I_{\vec{g},q}$ and $I_{\vec{g},p}$ and $k(\vec{g},q,a)=k(\vec{g},p,f_{p,q}(a))$ for all $a\in I_{\vec{g},q}$. \qed
\end{claim}

Given $\vec{a}=(a_p)_{p\in D}\in A_\III$, we define
\begin{equation*}
 \vec{g}_{\vec{a}} ~ = ~ (x_{p,a_p})_{p\in D} ~ \in ~ \prod_{p\in D}G_p. 
\end{equation*}
It is easy to see that $\vec{g}_{\vec{a}}$ is an element of the inverse limit $G_{\III_{ab}}$. 
As in the proof of Theorem \ref{theorem:FreeGroupFromSystem}, we can use the above claims to show that $G_{\III_{ab}}$ is 
a free abelian group with basis $\Set{\vec{g}_{\vec{a}}}{\vec{a}\in A_\III}$.
\end{proof}


\section{Good inverse systems of sets}\label{section:system}

In this section, we complete the proofs of the results listed in Section \ref{section:Introduction} by construction \emph{suitable} inverse systems of sets from the assumptions appearing in the statements of those results. The next definition precises the notion of \emph{suitable inverse system}.

\begin{definition}\label{definition:GoodInverseSystem}
 Let $\lambda$ and $\nu$ be infinite cardinals. We say that an inverse system $\III=\langle\seq{A_p}{p\in D},\seq{f_{p,q}}{p,q\in D,~p\leq_\DDD q}\rangle$ of sets 
over a directed set $\DDD=\langle D,\leq_\DDD\rangle$ is $(\lambda,\nu)$-good if the following statements hold.
\begin{enumerate}
 \item Player II has no winning strategy in $\calG(\DDD)$.

 \item $\betrag{D}\leq\lambda$ and $\betrag{A_p}\leq \lambda$ for all $p\in D$.

 \item $\betrag{A_{\III}}=\nu$.
\end{enumerate}
\end{definition}

The following proposition summarizes the results of the previous sections.

\begin{proposition}\label{proposition:FieldsFromGoodSystems}
 If $\III$ is a $(\lambda,\nu)$-good inverse system of sets and $p$ is either $0$ or a prime number, 
 then there are fields $K_0$ and $K_1$ of characteristic $p$ and cardinality $\lambda$ 
 with the property that $\Aut{K_0}$ is a free group of cardinality $\nu$ and $\Aut{K_1}$ is a free abelian group of cardinality $\nu$.
\end{proposition}

\begin{proof}
 Our assumptions imply that $A_\III\neq\emptyset$ and all groups appearing in the corresponding inverse systems of groups $\III_{gr}$ and $\III_{ab}$ 
 have cardinality at most $\lambda$. 
 We can now apply Theorem \ref{theorem:ModelIAut} to find fields $K_0$ and $K_1$ of characteristic $p$ and cardinality $\lambda$ 
 such that the group $\Aut{K_0}$ is isomorphic to the inverse limit $G_{\III_{gr}}$ and the group $\Aut{K_1}$ is isomorphic to the inverse limit $G_{\III_{ab}}$. 
 By our assumptions, Theorem \ref{theorem:FreeGroupFromSystem} implies that $G_{\III_{gr}}$ is a free group of cardinality $\nu$ and 
 Theorem \ref{theorem:AbelianFreeGroupFromSystem} implies that $G_{\III_{ab}}$ is a free abelian group of cardinality $\nu$.
\end{proof}

In order to prove Theorem \ref{theorem:Main1}, we now construct $(\lambda,2^\lambda)$-good inverse system from the assumption $\lambda=\lambda^{\aleph_0}$.

\begin{proposition}\label{proposition:ClosedDirectedSet}
 Let $\DDD=\langle D,\leq_\DDD\rangle$ be a directed set with the property that for every $P\in[D]^{\aleph_0}$ there is a $q\in D$ with $p\leq_\DDD q$ for all $p\in P$. 
Then Player I has a winning strategy in $\calG(\DDD)$ and hence Player II has no winning strategy in $\calG(\DDD)$. \qed
\end{proposition}

\begin{lemma}\label{lemma:GoodInverseSystemFromAssumption}
If $\lambda$ is a cardinal with $\lambda=\lambda^{\aleph_0}$, then $\III_\lambda$ is a $(\lambda,2^\lambda)$-good inverse system of sets.
\end{lemma}

\begin{proof}
 By Proposition \ref{proposition:ClosedDirectedSet}, Player II has no winning strategy in $\langle[\lambda]^{\aleph_0},\subseteq\rangle$. 
 It is shown in Example \ref{example:LambdaCountableInverseLimit} that the direct limit $A_{\III_\lambda}$ has cardinality $2^\lambda$. 
 The other cardinality requirements follow directly from the assumption $\lambda=\lambda^{\aleph_0}$.
\end{proof}

The statements of Theorem \ref{theorem:Main1} and Theorem \ref{theorem:Main2} now follow directly from the combination of Proposition \ref{proposition:FieldsFromGoodSystems} 
and Lemma \ref{lemma:GoodInverseSystemFromAssumption}. Next, we show that Corollary \ref{corollary:CHSCHAnswer} is a direct consequence of Theorem \ref{theorem:Main1} and the 
results presented in the first two sections.

\begin{proof}[Proof of Corollary \ref{corollary:CHSCHAnswer}]
 If $\lambda$ is an infinite cardinal with $\cof{\lambda}>\omega$, then our assumptions and {\cite[Theorem 5.22]{MR1940513}} imply $\lambda=\lambda^{\aleph_0}$ and 
we can apply Theorem \ref{theorem:Main1} to find a field of cardinality $\lambda$ with the desired properties.

Next, if $\lambda=\aleph_0$ or $\lambda$ is a singular strong limit cardinal of countable cofinality, then Theorem \ref{theorem:ST1} 
and Theorem \ref{theorem:ST2} imply that the automorphism group of a field of cardinality $\lambda$ either has cardinality at most $\lambda$ or is not a free group.

Finally, let $\lambda$ be a singular cardinal and $\kappa<\lambda$ be an uncountable regular cardinal with $2^\kappa>\lambda$. By the above arguments, we can apply 
Theorem \ref{theorem:Main1} to find a field $K$ of cardinality $\kappa$ whose automorphism group is a free group of cardinality $2^\kappa$. Then we can construct 
a first-order language $\calL$ of cardinality $\lambda$ and an $\calL$-model $\calM$ whose automorphism group is isomorphic to $\Aut{K}$. By the results presented 
Section \ref{section:MI}, this allows us to produce a field of cardinality $\lambda$ with the desired properties.
\end{proof}

The following proposition will allow us to prove Theorem \ref{theorem:OuterModelFields}.

\begin{proposition}\label{proposition:CoverInnerNiceSystem}
 Let $M$ be an inner model of $\ZFC$ with the property that every countable set of ordinals in $\VV$ is contained in a set that is an element of $M$ and countable in $M$. 
 If $\lambda$ is a cardinal with $\lambda=(\lambda^{\aleph_0})^M$, then there is a $(\lambda,\nu)$-good inverse system for some cardinal $\nu$ 
 with $(2^\lambda)^M\leq\nu\leq 2^\lambda$.
\end{proposition}

\begin{proof}
 Set $D=([\lambda]^{\aleph_0})^M$ and $\DDD=\langle D,\subseteq\rangle$. Given a sequence $\seq{u_n\in D}{n<\omega}$, our assumption implies that there is a 
 $u\in D$ with $\bigcup_{n<\omega}u_n\subseteq u$ and Proposition \ref{proposition:ClosedDirectedSet} shows that Player II has no winning strategy in $\calG(\DDD)$.
 Let $\III=\III_\lambda^M$ and $\nu$ be the cardinality of $A_\III$. Then every element of $({}^\lambda 2)^M$ gives rise to a distinct element of $A_\III$ and 
 $\nu$ is an infinite cardinal greater than or equal to $(2^\lambda)^M$. We can conclude that $\III$ is $(\lambda,\nu)$-good.
\end{proof}

Since partial orders of the form $\Add{\omega}{\kappa}$ satisfy the countable chain condition 
and therefore every countable set of ordinals in an $\Add{\omega}{\kappa}$-generic extension of the ground model is covered by a set countable set of ordinals 
from the ground model, we can directly derive the statement of Theorem \ref{theorem:OuterModelFields} from Proposition \ref{proposition:FieldsFromGoodSystems} 
and Proposition \ref{proposition:CoverInnerNiceSystem}.

\begin{lemma}\label{lemma:GoodSystemFromTree}
 Let $\lambda$ be a cardinal of uncountable cofinality and $\TTT$ be a tree of cardinality and height $\lambda$ with the property that the set $[\TTT]$ of cofinal branches 
through $\TTT$ has infinite cardinality $\nu$. Then $\III_\TTT$ is a $(\lambda,\nu)$-good inverse system of sets.
\end{lemma}

\begin{proof}
 By Proposition \ref{proposition:ClosedDirectedSet}, the assumption $\cof{\lambda}>\omega$ implies that Player II has no winning strategy in $\calG(\langle\lambda,\leq\rangle)$. 
 The computations in Example \ref{example:TreesInverseSystems} show that the inverse limit $A_{\III_\TTT}$ also has cardinality $\nu$. 
 Since the cardinality requirements of Definition \ref{definition:GoodInverseSystem} are obviously satisfied, this completes the proof of the lemma.
\end{proof}

The statement of Theorem \ref{theorem:InnerModelFields} now follows directly from Proposition \ref{proposition:FieldsFromGoodSystems} and Lemma \ref{lemma:GoodSystemFromTree}.


\section{Some questions}

We close this paper with questions raised by the above results.

\begin{question}
 Is it consistent with the axioms of $\ZFC$ that there is a cardinal $\lambda$ of uncountable cofinality with the property that no free group of cardinality $2^\lambda$ is isomorphic to the automorphism group of a field of cardinality $\lambda$?
\end{question}

\begin{question}
 Is it consistent with the axioms of $\ZFC$ that there is a cardinal $\lambda$ of uncountable cofinality with the property that no free group of cardinality greater than 
$\lambda$ is isomorphic to the automorphism group of a field of cardinality $\lambda$?
\end{question}

\begin{question}
 Is it consistent with the axioms of $\ZFC$ that there is a singular cardinal $\lambda$ of uncountable cofinality with the property that there is no tree of cardinality and 
height $\lambda$ with more than $\lambda$ many cofinal branches?
\end{question}



\end{document}